\documentclass[12pt]{article}
\usepackage{amsmath, amssymb,amsthm}
\usepackage[pdftex]{graphicx}
\usepackage[T1]{fontenc}
\usepackage{booktabs}

\theoremstyle{definition}

\newtheorem{lemma}{Lemma}

\newtheorem{thm}{Theorem}
\newtheorem{example}{Example}
\newtheorem{cor}{Corollary}
\newtheorem{rem}{Remark}

\newcommand{\Y}{\textbf{Y}}
\newcommand{\y}{\textbf{y}}
\newcommand{\X}{\textbf{X}}
\newcommand{\EE}{\mathbb{E}}
\newcommand{\A}{\mathcal{A}}
\newcommand{\FF}{\mathcal{F}}
\newcommand{\LL}{\mathcal{L}}

\newcommand{\RR}{\mathbb{R}}
\newcommand{\II}{\mathbb{I}}

\newcommand{\var}{\mathbb{V}\text{ar}}

\begin{document}


\begin{center}
\textbf{\large On continuous distribution functions, minimax and best invariant estimators, and integrated balanced loss functions  \footnote{\today}}

 {\sc
Mohammad
Jafari Jozani$^{a,}$\footnote{Corresponding author: m$_{-}$jafari$_{-}$jozani@umanitoba.ca},  Alexandre Leblanc$^{a}$ and 
 \' Eric
 Marchand$^{b}$,
} \\

{\it a University of Manitoba, Department of Statistics, Winnipeg, MB, CANADA, R3T 2N2 } \\

 {\it b  Universit\'e de
    Sherbrooke, D\'epartement de math\'ematiques, Sherbrooke, QC,
    CANADA, J1K 2R1} \\

\end{center}

\abstract{
We consider the problem of estimating a continuous distribution function $F$, as well as meaningful functions
$\tau(F)$ under a large class of loss functions. We obtain best invariant estimators and establish their
minimaxity for H\"{o}lder continuous $\tau$'s and strict bowl-shaped losses with a bounded derivative. 
We also introduce and motivate the use of integrated balanced loss functions which combine the criteria
of an integrated distance  between a decision $d$ and $F$, with the proximity of $d$ with a target
estimator $d_0$.  Moreover, we show how the risk analysis of procedures under such an integrated balanced
loss relates to a dual risk analysis under an ``unbalanced'' loss, and we derive best invariant estimators,
minimax estimators, risk comparisons, dominance and inadmissibility results.   Finally, we expand on various illustrations and  
applications relative to maxima-nomination sampling, median-nomination  
sampling, and a case study related to bilirubin levels in the blood of babies suffering from  
jaundice.
 
}



%
\section{Introduction}
%

An appealing and wide ranging formulation for estimating a continuous distribution function (cdf) $F$ based on $X=(X_1, \ldots, X_n)$, where  $X_i$'s are independently and identically distributed (i.i.d.)  on $I=(a,b) \subseteq \mathbb{R}$ with cdf $F$, is to measure the discrepancy between an estimate
$d(\cdot):\mathbb{R} \to [0,1]$ and $F$ as 
\begin{equation}
\label{loss}
\int_{\mathbb{R}} \, \rho(d(t)-F(t)) \, H(F(t)) \, dF(t)\,,
\end{equation}
where $\rho$ is strictly bowl-shaped on its domain with $\rho(0)=0$, $\rho'(z)<0$ for $z<0$, $\rho'(z)>0$ for $z>0$, and
$H$ is a continuous and positive weight function.  Aggarwal (1955) introduced such a formulation for Cram\'er-von Mises loss with
$\rho(z)=|z|^r$; $r \in \{1,2, \ldots \}$, considered an invariance structure relative to  the group of continuous and strictly increasing transformations,
and obtained best invariant estimators of $F$.  For instance, the empirical distribution function $F_n$ is the best invariant estimator of $F$ under loss (\ref{loss})
with $\rho(z)=z^2$ and $H(z)=(z(1-z))^{-1}$ (e.g., Ferguson, 1967, Section 4.8).  Now, in terms of the larger class of (not necessarily invariant) procedures, challenging issues with regards to the potential minimaxity and admissibility of the best invariant procedure have been addressed by Dvoretzky et al.  (1956), Phadia (1973), Cohen and  Kuo (1985), Brown (1988), Yu (1989), and Yu and  Chow (1991).  
Namely, Yu (1992) established the minimaxity of the best invariant procedure in Aggarwal's setup and analog minimaxity findings have been obtained by 
Mohammadi and van Zwet (2002, entropy loss), Ning and Xie (2007, Linex loss), and St\k{e}pie\'n-Baran (2010, strictly convex $\rho$).  Parallel developments 
for the alternative Kolmogorov-Smirnov loss $\sup_{t \in \mathbb{R}} |d(t)-F(t)|$ were given by Brown (1988), Friedman et al.\ (1988), and 
Yu and Phadia (1992). 

In this paper, we seek to extend St\k{e}pie\'n-Baran's minimax result to loss functions of the form
\begin{equation}
\label{lossG}
L_{\rho,\tau}(d,F)\,=\, \int_{\mathbb{R}}  \, \rho(\tau(d(t))-\tau(F(t))) \, dF(t)\,,
\end{equation}
with $\rho$ strict bowl-shaped, differentiable almost everywhere (a.e.), and with $\tau(\cdot)$ a continuous and strictly monotone function on $[0,1]$.

A first motivation here is to provide analytical results applicable to non-strict convex choices of $\rho$ which are not covered by previous findings even for identity
$\tau$.  As well,  the  loss functions in \eqref{lossG} are flexible enough to include loss functions of the form
\begin{equation}
\label{losslogG}
 \int_{\mathbb{R}}  \, \rho_0(\frac{d(t)}{F(t)}) \, dF(t)\,,
\end{equation}
contrasting directly the ratios $\frac{d(t)}{F(t)}$, as opposed to the differences $d(t)-F(t)$, with $\rho \equiv \rho_0 \circ \log$, and $\rho_0$ strict bowl-shaped.
Notice here that the strict bowl-shapedness of $\rho$ and $\rho_0$ are equivalent, which is not the case as for convexity.  An example of (\ref{losslogG}) is  the integrated entropy loss with 
$\rho_0(z)=z^{-1} + \log(z) -1$, (see Mohammadi and van Zwet, 2002).
 The  losses in \eqref{lossG} also encompass integrated $L^2$ losses of the form 
 \begin{equation}
L_{\tau}(d,F) = \int_{\RR}  \big( \tau(d(t)) - \tau(F(t)) \big)^2  dF(t),
\label{eq.loss}
\end{equation}
which correspond of course to $\rho(z)=z^2$ in \eqref{lossG}.  An interesting example of  \eqref{eq.loss} is the so-called precautionary loss function with $\tau(z)=e^{az}$; $a \neq 0$; which is nicely motivated from a practical point of view (e.g.,  Sch\"abe, 1991; Norstr{\o}m, 1996). For more examples see Jafari Jozani and Marchand (2007).

Another motivation to study integrated losses
of the form (\ref{lossG}) with non-identity $\tau$ resides in the equivalence of the performances of estimates $d(\cdot)$ of $F$ under loss (\ref{lossG}) 
with estimates $d^*(\cdot) \equiv \tau(d(\cdot))$ of $\tau(F)$ under loss   
\begin{equation}
\label{lossGtau}
\int_{\mathbb{R}} \, \rho(d^*(t) -\tau(F(t))) \, dF(t)\,.
\end{equation}
Although the problems are mathematically equivalent, they emanate from different practical perspectives.  Indeed, for the latter problem,
our interest lies in estimating a meaningful function  $\tau(F(t)), t \in \mathbb{R}$, such as a logarithmic function $\log(1+z)$, polynomials $z^m$ and $1-(1-z)^m$ representing for instance the cdf's of the minimum and maximum of $m$ independent copies generated from $F$, and similarly $z^{1/k}$ and $1-(1-z)^{1/k}$ arising in maxima or minima nomination samples when the set size is  an integer $k\geq 1$ (e.g., Wells and Tiwari, 1990).
Other interesting choices, further discussed in Examples \ref{odds}, \ref{oddsc},  and \ref{minimaxapplications}, are the odds-ratio $\tau(z)=\frac{z}{1-z}$ and the log odds-ratio $\tau(z)=\log(\frac{z}{1-z})$.  However, even in cases where a best invariant estimator exists, these choices will not satisfy a H\"{o}lder continuity condition on $\tau$ that is required for the minimaxity of the best invariant estimator to follow from our Theorem \ref{thm.minimax}.

In Section 2.1, we provide preliminary results and examples for the best invariant estimator, expand on issues related to the role of the action space, the presence of best invariant solutions which are not genuine cdf's, and corresponding adjustments
which we present as best constrained invariant estimators of $F$ and $\tau(F)$ (Remark \ref{genuine}).  In Section 2.2, we pursue with a general minimax result (Theorem \ref{thm.minimax}).  To this end, we exploit a key result from Yu and Chow (1991), we require $\rho$ to have a bounded derivative, and we work with a H\"{o}lder continuity assumption for $\tau$.    This minimax result can be viewed as an extension of St\k{e}pie\'n-Baran's (2010) minimax result to losses $L_{\rho, \tau}(d,F)$ with either strict bowl-shaped $\rho$ and/or non-identity $\tau$.  We also point out  (Theorem \ref{minimax-weighted}) that the best invariant and minimax properties  are  preserved for a class of weighted integrated loss functions, which will play a critical role in Section 3. 

In Section 3, as an alternative, we propose and motivate the use of an integrated balanced loss function in the spirit of Jafari Jozani, Marchand and Parsian (2006).  This loss function, presented in the context of estimating $\tau(F)$, is of the form 
$$
 L_{\omega,d_0}(d,F)= \int_{\RR} \{w(x,t)
(d(t)-d_0(t))^2 + (1-w(x,t)) (d(t)-\tau(F(t))^2 \} 
\, dF(t)\,$$
with $d_0$ being the target estimator of $\tau(F)$, and $w(\cdot, \cdot) \in [0,1]$ is a data dependent  weight function which permits one to combine the criteria that the estimate $d(\cdot)$ be close to the target  estimator  $d_0(\cdot)$ (which can be chosen for instance as $\tau(F_n)$, with $F_n$ being the empirical cdf)  with  integrated squared error $L_{\tau}(\tau^{-1}(d), F)$ as in \eqref{eq.loss}.  We describe explicitly how the performance of estimators of $\tau(F)$ under loss $L_{\omega,d_0}$ relates to the performance of a dual estimator of $\tau(F)$ under ``unbalanced'' loss $L_{\omega,d_0}$ with $\omega\equiv0$.  This leads to the determination of the best invariant estimator (Theorem \ref{blfbestinvariant}), as well as a proof of its minimaxity (Theorem \ref{blfminimax}) among all estimators for cases where both $w$ and $d_0$ satisfy an invariance requirement (i.e., being functions of the $X_i$'s only through their order statistics).  Moreover, the same duality between the ``balanced'' and ``unbalanced'' cases, along with known results for the ``unbalanced'' case leads to dominance and inadmissibility results (Theorem \ref{dominance}).
We advocate the use of such balanced integrated losses to provide a flexible and natural tool for estimating $F$.  In particular, it permits us to set the weight  $w(x, t)$ equal to $1$ whenever $F_n(t)$ takes the values $0$ or $1$, leading to a  best invariant (and minimax) estimator that is  a genuine cdf.  

Section 4 is devoted to  
applications and illustrations relative to   
maxima-nomination sampling and median-nomination sampling.   In Section 5,   an actual  
data set, pertaining to bilirubin levels in the blood of babies  
suffering from jaundice, is analyzed via an integrated balanced loss  
function.  In Section 6, we provide some concluding remarks.  Finally,    the proofs and  further complementary developments  
with respect to balanced loss functions are presented in the Appendix.

%
\section{Best invariant and minimax estimators of $F$ and $\tau(F)$}
%

\subsection{Preliminary results and the best invariant estimator}

Let $\X = (X_1, \ldots, X_n)$  be a random sample of size $n \geq 2$ from an unknown continuous distribution function $F$ supported on $(a,b)$,  
and denote 
 its associated order statistics by 
$\Y = (Y_1, \ldots,Y_n)$. Define also $Y_0=a$ and $Y_{n+1}=b$. Let
$
\mathcal{A}=\{ d(\cdot) : d (\cdot)$  is a nondecreasing function from  $\mathbb{R} \text{ onto } [0,1] \}
$
be the action space, and 
$
\FF = \{ F: F$  is a continuous cumulative distribution function on $ \mathbb{R} \}
$
be the parameter space.  Consider estimating $F$ under the  integrated loss $L_{\rho,\tau}(d, F)$ in \eqref{lossG},
$\rho$ strict bowl-shaped and differentiable a.e., and assume
without loss of generality that $\tau$ is strictly increasing (otherwise, transform $\tau$ to $-\tau$).
For an estimator $d(X;\cdot)$ of $F$, we define the corresponding frequentist risk as
$R_{\rho,\tau}(d(X;\cdot), F)= E_F L_{\rho,\tau}(d(X;\cdot),F)$.

In his seminal paper, Aggarwal~(1955) showed that, under the group of continuous and strictly increasing
transformations,  the class of invariant estimators considered here
leads to estimators which are nondecreasing step functions with jumps at the observed order statistics, in
other words,  of the form
\begin{equation}
d(\Y; t) = \sum_{i=0}^n u_i\, \II(Y_i \leq t < Y_{i+1}),
\label{eq.invest}
\end{equation}
for $t\in(a, b)$, where $0 \leq u_0  \leq \ldots \leq u_n \leq 1$, and $\II(B)$ denotes the indicator function of a set $B$.
Our next results identify the best invariant estimator of
$F$ in the current setup.  Here and throughout, we set $T_i$, $i=0, \ldots, n$, to be random variables such that
$$ T_i \sim \text{Beta}(i+1,n-i+1)\,,\; \hbox{  with pdf } f_{T_i}(\cdot)\,.$$
 
\begin{thm}
\label{thm.best}
A best invariant estimator of $F$, whenever it exists, under loss $L_{\rho,\tau}(d,F)$ in (\ref{lossG}), is given by $
d^*(\Y; t) = \sum_{i=0}^n u^*_i \II(Y_i \leq t < Y_{i+1}),
$
where
$u^*_i$ is the Bayes point estimate of $p$ for the model $B|p \sim \hbox{Bin}(n,p)$, the observed $B=i$, the prior
$p \sim U(0,1)$ (i.e., posterior for $p$ is $\text{Beta}(i+1,n-i+1)$), and loss
$L(d,p)=\rho(\tau(d) -\tau(p))$.   The risk of  $d^*(\Y, t)$ is constant in $F$ and given by
$$ R_{\rho,\tau}(d^*, F)= \frac{1}{n+1} \; \sum_{i=0}^n  \int_0^1 \rho(\tau(u_i^*) -\tau(t)) \, f_{T_i}(t) \, dt, $$
where 
$ 0 < u_0^*  \leq \ldots \leq u_n^* < 1. $
\end{thm} 
\begin{proof}
The proof is given in Section  \ref{thm1}.
\end{proof}

\begin{rem}\label{rem1}
A more general representation holds for Theorem \ref{thm.best} in the presence of a weight $H$ as in (\ref{loss}), with  the  $u_i^*$'s defined similarly but with a  prior 
density on $p$ that is proportional to  $H(p)$.  
\end{rem}

\begin{rem}
\label{remark2}
Since $\tau$ is strictly monotone and continuous, a best invariant estimator of $\tau(F)$ under loss $L_{\rho,\tau}(\tau^{-1}(d),F)$ is given by $
d^*_{\tau}(\Y; t) = \sum_{i=0}^n \tau(u^*_i) \, \II(Y_i \leq t < Y_{i+1}),
$ with the $u^*_i$'s given in Theorem \ref{thm.best}.
\end{rem}

For the particular case where $\rho(z)=z^2$ is squared error loss, since Bayes estimators are posterior expectations, the following specialization of Theorem 
\ref{thm.best} becomes immediately available. 
\begin{cor}
 \label{bestsel}
\begin{enumerate}
\item[ (a)]  A best invariant estimator  of $\tau(F)$ under loss $L_{\tau}(\tau^{-1}(d),F)$ (see (\ref{eq.loss})), whenever it exists,   
is given by $d^*_{\tau}(\Y; t)=\sum_{i=0}^n u^*_{i, \tau} \II(Y_i \leq t < Y_{i+1})$, 
where  $u^*_{i, \tau}= \tau(u^*_i) = \EE \big[ \tau(T_i) \big]$ for $i=0, \ldots,n$;
\item[ (b)]
A best invariant estimator of $F$ under loss $L_{\tau}(d,F)$ in (\ref{eq.loss}), whenever it exists, is given by $
d^*(\Y; t) = \sum_{i=0}^n u^*_i \II(Y_i \leq t < Y_{i+1}),
$
where
\begin{equation}
u^*_i = {\tau}^{-1} \big( \EE \big[ \tau(T_i) \big] \big), \qquad \qquad \text{for } i=0, \ldots,n.
\label{eq.uibest}
\end{equation}
\end{enumerate}
In both cases, the risk of the best invariant estimator is constant and given by $\frac{1}{n+1} \sum_{i=0}^n \var\big[\tau(T_i)
\big].$

\end{cor}

\begin{example}
\label{powers}
Corollary \ref{bestsel} applies to powers of $F$ with $\tau(z)=z^m$, $m>0$, and simply brings into play corresponding
moments for Beta distributed $T_i$'s.  For instance with $\rho(z)=z^2$, we obtain in Corollary \ref{bestsel}
\begin{align}\label{general-u}  u_i^*=\{E({T_i}^m\,)\}^{1/m} = \left(\frac{(n+1)! \, \Gamma(i+m+1)}{i! \,\Gamma(n+m+2)}  \right)^{1/m}\,,\, i=0,\ldots,n. 
\end{align} 
When $m=1$,  a best invariant estimator of $F$ under
loss~\eqref{eq.loss} is obtained when
$u_i^* = E(T_i) = (i+1)/(n+2)$,
a result first obtained by Aggarwal~(1955). 
\end{example}

\begin{example} (Odds and log-odds ratios)
\label{odds}
For the situation where $\tau(F)=\frac{F}{1-F}$ and $\rho(z)=|z|$ in (\ref{lossG}), the risk of any invariant procedure is infinite as seen from (\ref{s}) with $i=n$ and the divergence of $\int_0^1 |\tau(u) - \frac{t}{1-t}| \, f_{T_n}(t) \, dF(t)$, for any $\tau(u)$.  The same is true for $\rho$'s that are convex on $(0,\infty)$, such as for $L^p$ integrated losses with $\rho(z)=|z|^p$, $p>1$.   Alternatively, concave $L^p$ choices with $0<p<1$ will lead to the existence of a best invariant estimator as can be verified by the convergence of (\ref{s}) for all $i$, and with $\tau(u)=0$ (for instance).  For estimating $\tau(F)=\log(\frac{F}{1-F})$, the best invariant procedures
will exist in many more cases.  In particular for $\rho(z)=z^2$, the best invariant procedures of Corollary \ref{bestsel} do exist with 
$u_{i,\tau}^*= E\left[ \,\log(\frac{T_i}{1-T_i}) \, \right]\,, $ 
and $u_i^*= {e^{u_{i,\tau}^*}} / {(1+e^{u_{i,\tau}^*})}; \; i=0,\ldots, n\,.$
\end{example}

\begin{rem}
\label{genuine}
When $b < \infty$, all estimators of the form $\sum_{i=0}^{n} u_i\,   
\mathbb{I}(Y_i \leq t < Y_{i+1}) + u_{n+1}\mathbb{I}(t \geq b)$,
with fixed common $u_0, \ldots, u_{n}$ and different $u_{n+1}$ are  
equivalent under loss (2).  Hence, there are many best invariant  
estimators
in the context of Theorem 1, and we can select $u_{n+1}=1$ so that  
best invariant estimates behave like a genuine cdf in the right tail.   
A similar situation
applies when $a > -\infty$.  When $a = -\infty$ and $b = +\infty$, the  
best invariant estimator is unique as given by Theorem 1.

A best invariant estimator of $F$ under loss $L_{\rho,\tau}(d,F)$
is always such that $u_0^*>0$ and $u_n^*<1$ (\textit{cf.} Theorem 1).  
Along with the observations of the previous paragraph, this implies  
that $d^*$ can never be a
genuine distribution function on the real line whenever $a =-\infty$  
or $b=+\infty$. A simple way of overcoming such a  difficulty is to force the invariant
estimator of $F$ in ~\eqref{eq.invest} to take the values $u_0^* =0$ 
and $u_n^*=1$.
Said  otherwise, one may work with the constrained action space
$
\A_c = \{ d(\cdot) : d \text{ is a distribution function on } \mathbb{R} \}.\,
$
Since the minimization is performed for each step $i$, it is immediate 
that the best invariant estimator of $F$ for
such a constrained problem under loss $L_{\rho,\tau}(d,F)$ is given by
$d_c^*(\Y; t) = \sum_{i=1}^{n-1} u^*_i\,  \II(Y_i \leq t < Y_{i+1}) + 
\II(t \geq Y_n)\,,$
where
$u^*_i = \tau^{-1} \big( \EE \big[ \tau(T_i) \big] \big)$,
for $i=1, \ldots,n-1$.
\end{rem}
\begin{example}(Example \ref{odds} continued)  
\label{oddsc}
Revisiting Example \ref{odds} with $\tau(F)=\frac{F}{1-F}$ and $\rho(z)=|z|$, a constrained best invariant estimator
of $F$ will exist, is derived from \eqref{ss}, leading to 
$u^*_i$ being  the median of $T_i \sim \hbox{Beta}(i+1,n-i+1)$, for  $i=1, \ldots,n-1$.  
\end{example}

\subsection{Minimaxity of the best invariant estimator}

We now consider the minimaxity of the best invariant estimator $d^*$ introduced in Theorem~\ref{thm.best} among
all estimators in $\A$.  To this end, we need the following useful lemma which establishes the existence of
an invariant estimator $d_0$ and a cdf $F_0$ under which the  behaviour of  $d_0$ is arbitrarily close to that of a given $d \in A$. 

\begin{lemma} (Yu and  Chow, 1991, Theorem 4) \label{lemma.yuchow}
Suppose that $d=d(\Y; t)$ is a nonrandomized estimator with finite risk and a measurable function of the
order statistics $\Y$. For any $s, \delta>0$ there exists a uniform distribution $P_0$  on a Lebesgue
measurable subset $I\subseteq \mathbb{R}$ and an invariant estimator $d_0 \in \mathcal{I}$
such that 
\begin{equation*}
P_0^{n+1}\{ (\Y, t) : | d(\Y; t) -d_0(\Y; t)| \geq s\} \leq \delta,
\end{equation*}
where $n \geq 2$ corresponds to the sample size.
\end{lemma}

The following result extends Theorem~2.2 of Yu~(1992) and Theorem~1 of
St\k{e}pie\'n-Baran (2010) to the class of losses $L_{\rho,\tau}(d,F)$, when $\tau$ is H\"{o}lder continuous of
order $\alpha \in (0,1]$, that is, there exists constants $\alpha,M>0$ such that
\begin{equation*}
\big| \tau(t_1)-\tau(t_2)\big| \leq M \, \big| t_1-t_2 \big|^{\alpha},
\end{equation*}
for all $t_1, t_2 \in [0,1]$. We write $\tau \in \LL(\alpha)$ to denote this.  Note that, under the H\"{o}lder continuity
assumption for $\tau$ and the boundedness of $\rho$ on any finite interval, the risk of any invariant estimator 
is finite (hence a best invariant estimator will exist) as seen by Theorem 1's representation (\ref{s}).

\begin{lemma}\label{thm.eps}
Consider estimating $F$ under loss (\ref{lossG}) with $\rho$ differentiable, strict-bowl shaped, $\rho'$ bounded, and
$\tau \in \LL(\alpha)$ for $\alpha \in (0,1]$. Then, for any $d \in \A$ and $\epsilon >0$, there exists
$F_0 \in \FF$ and $d_0 \in \mathcal{I} $ such that $| R(d,F_0) - R(d_0,F_0)| \leq \epsilon \,. $
\end{lemma}
\begin{proof}
The proof is given in Section \ref{lem2}
\end{proof}

What follows is our main minimaxity  result.

\begin{thm}\label{thm.minimax}
For the problem of estimating $F$ under loss (\ref{lossG}) with $\rho$ differentiable, strict-bowl shaped, $\rho'$ bounded, and $\tau \in \LL(\alpha)$ for $\alpha \in (0,1]$, the best invariant estimator $d^*$ is minimax,  that is  $$\inf_{d \in \A} \sup_{F \in \FF} R_{\tau}(d,F) = \sup_{F \in \FF} R_{\tau}(d^*,F).$$
\end{thm}
\begin{proof}
The proof is given in Section \ref{thm2}
\end{proof}

\begin{example}
\label{minimaxapplications}
As a continuation of Examples \ref{powers} and \ref{odds}, we summarize how the results of this section apply or don't apply.
For log-odds ratios,  although a best invariant estimator exists, the above minimaxity result does not apply since the
function $\tau(z)=\log(\frac{z}{1-z})$ does not satisfy the H\"{o}lder continuity assumption.  For powers $\tau(z)=z^m$ with $m>0$,
we have H{\"o}lder continuity for $0<\alpha \leq m$ and the corresponding best invariant estimators of $F$ are minimax as
long as $\rho$ satisfies the given conditions (examples include $L^p$ with $\rho(z)=z^p;\, p \geq 1$; Linex with
$\rho(z)=e^{az}-az-1$, $a \neq 0$, among others).   Equivalently,  Remark \ref{remark2}'s best invariant estimator of
$\tau(F)=F^m$, under loss $L_{\tau}(\tau^{-1}(d),F)$, is also minimax by virtue of Theorem \ref{thm.minimax}.    
\end{example}

We conclude this section by expanding upon best invariant estimators and their minimaxity, for a more general class of weighted integrated loss functions given by
\begin{equation}
\label{losswgtau}
 L_{w_n, \rho,\tau}(d, F)= \int_{\mathbb{R}} w_n(t) \; \rho\big(\tau(d(t))- \tau(F(t)\big)\, dF(t)\,,
\end{equation} 
where the conditions on $\rho$ and $\tau$ are as above, and where $w_n(\cdot)$ is an invariant weight function, i.e. such that $w_n(t)= w_i$ when $t\in[Y_i, Y_{i+1})$, $i=0, \ldots, n$, with constants
$0<w_i \leq 1$.   In fact, the procedure 
obtained in  Theorem \ref{thm.best} is also the best invariant and minimax estimator of $F$ for such loss functions.  This is a key result that will prove to be quite useful for the integrated balanced loss functions developments of Section 3 below. 

\begin{thm}\label{minimax-weighted} 
The estimator $d^*(\Y; t) = \sum_{i=0}^n u^*_i \II(Y_i \leq t < Y_{i+1})$
given in Theorem \ref{thm.best} is best invariant and minimax for loss  $L_{w_n, \rho,\tau}(d, F)$ as in (\ref{losswgtau}).
\end{thm}
\begin{proof}
The proof is given in Section \ref{thm3}.
\end{proof}

%
\section{Integrated balanced loss functions}
%

We now introduce and advocate the use of integrated balanced loss functions of the form
\begin{equation}
\label{iblf} L_{\omega,d_0}(d,F)= \int_{\RR} \{w(x,t)
(d(t)-d_0(t))^2 + (1-w(x,t)) (d(t)-\tau(F(t))^2 \} 
\, dF(t)\,,
\end{equation}
where $d_0(t)$ is a target estimate of $\tau(F(t))$, such as $\tau(F_n(t))$ with $F_n$ the empirical cdf, and $w(\cdot,\cdot) \in (0,1]$ is a possibly data
dependent weight function.  In
the spirit of Jafari Jozani, Marchand and Parsian (2006), this
integrated balanced loss function allows one to combine the  desire of closeness of an estimator  $d(X,\cdot)$ to both: {\bf  (i)}  the target estimator $d_0(X,\cdot)$ and {\bf (ii)} the unknown function $\tau(F(\cdot))$.
We provide below analysis for integrated balanced loss functions as in
(\ref{iblf}), which is unified with respect to the choices of $w$,
$d_0$, and $\tau$.  For ease of notation, we hereafter write $w$ instead of
$w(\cdot,\cdot)$ or $w(X,\cdot)$,  unless emphasis is required.
Although, we do proceed with developments for the general
situation, we will focus on particular cases where $d_0$ and $w$
are invariant (with respect to monotone transformations of the
data points) and hence expressible as $d_0(y,\cdot)$ and
$w(y, \cdot)$ without loss of generality.   For invariant $d_0$ and $w$, we derive
the best invariant procedure and show that it is minimax for $\tau \in \LL(\alpha)$, thus
extending the ``unbalanced'' loss (denoted $L_0$) result of Theorem \ref{minimax-weighted}
to an integrated balanced loss minimax result (Theorem \ref{blfminimax}). An
interesting feature will arise : if we choose $d_0$ and $w$ as
invariant, $d_0$ as a genuine cdf, and $w(y,t)=1$
whenever $F_n(y,t) \in \{0,1\}$, the best invariant procedure
$d_{w}^*(y,t)$ will coincide with $d_0(y,t)$ for $t \notin
[y_{1},y_{n}]$, and will therefore possess the potential advantage of
being a genuine cdf.

One can exploit Ferguson's decomposition to derive the best invariant
 estimator $d_{w}^*(Y,\cdot)$ for integrated balanced loss $L_{w,d_0}(d,F)$,
or for its associated risk
\begin{equation}
\label{blfrisk} R_{w,d_0}(d(X,\cdot),F)=E_F
[L_{\omega,d_0}(d(X,\cdot),F)], \end{equation} but we proceed
alternatively with a useful and general representation (Lemma
\ref{99}) of the risk $R_{w,d_0}$ in terms of weighted
unbalanced versions $R_H$, which will  be critical for
establishing the minimaxity of $d_{w}^*(Y,\cdot)$ (for invariant
$d_0$ and $w$), and also lead to further implications with regards
to admissibility and dominance.  Below, we represent estimators
$d(X,\cdot) \in \cal{A}$ as $d(x,t)=d_0(x,t) + (1-w(x,t)) g(x,t)$,
$x \in \mathbb{R}^n, t \in \mathbb{R}$.  The following now relates the risk
performance of such an estimator $d=d_0+(1-w)g$ of $\tau(F)$ under
loss $L_{\omega,d_0}$ to the performance of $d_0 + g$
under risk $R_H$ relative to an integrated weighted
squared error loss.\footnote{For convenience, we have dropped the subscript $\tau$ under $d$.}

\begin{lemma}
\label{99}
We have for $d_0(X,\cdot) \in \mathcal{A},  F \in \mathcal{F}$,
\begin{equation}
\label{blfrep} R_{w,d_0}(d_0+(1-w)g, F)
= R_{H_1}(d_0,F) + R_{H_2}(d_0 + g,F)\,,
\end{equation}
where $R_{H_1}$ and $R_{H_2}$ are risks associated to the losses
  $\int_{\mathbb{R}} \, H_i(w(x,t)) \, (d(t)-\tau(F(t)))^2 \,\, dF(t)$, $i=1,2$,
with $H_1(z)=z(1-z)$ and $H_2(z)=(1-z)^2$.
\end{lemma}
\begin{proof}
The proof is given in Section \ref{lem3}.
\end{proof}

\noindent  Now, by virtue of representation (\ref{blfrep}) where
the risk under integrated balanced loss of an estimator is
expressed in terms of the unbalanced risk $R_{H_2}$ of another
estimator, we obtain the following implications.

\begin{thm}
\label{blfbestinvariant} For invariant $d_0$ and $w(>0)$, the best
invariant estimator of $\tau(F)$, as long as it exists, under loss (\ref{iblf}) is
(uniquely) given by:
$$ d_{w}^*(y,t) = w(y,t) \, d_0(y,t) \, + \, (1-w(y,t))\, d_0^*(y,t)\,,$$
where $d_0^*$ is the best invariant estimator of $\tau(F)$ under
unbalanced loss $$L_0(d,F)= \int_{\RR} \, (d(t)-\tau(F))^2 \,
dF(t),$$ given in Corollary 1. 
\end{thm}
\begin{proof}
The proof is given in Section \ref{thm3}.
\end{proof}

We thus obtain an appealing representation, for invariant $d_0$ and $w$, of the optimal invariant estimate
$d_w^*(y,t)$ as a convex linear combination of the target estimate $d_0(y,t)$ and
the unbalanced best invariant estimate $d_0^*(y,t)$.  Now, consider the issue of whether or not $d_w^*$ is
a genuine cdf for the identity case $\tau(F)=F$ supported on $\mathbb{R}$. 
First, notice that we can force $\lim_{t \to -\infty}
d_w^*(y,t)=0$ and $\lim_{t \to \infty} d_w^*(y,t)=1$, for any fixed $y$, by selecting $d_0$ and $w$ such that
$d_0$ is a genuine cdf (hence $\lim_{t \to -\infty} d_0(y,t)=0$ and $\lim_{t \to \infty} d_0(y,t)=1$) and 
$w(y,t)=1$ whenever $F_n(y,t) \in \{0,1\}$.  The monotonicity of $d_w^*$ is still not necessarily guaranteed with such
choices of $d_0$ and $w$.  However, denoting $d_0(y,t)=\sum_{i=0}^{n}
u_i \, \II(y_i \leq t <y_{i+1})$ and $d_0^*(y,t)=\sum_{i=0}^{n}
u_i^* \, \II(y_i \leq t <y_{i+1})$, it is easy to see that the condition $\min(u_{i+1},u_{i+1}^*) \geq \max(u_{i},u_{i}^*)$ for all
$i$ forces $d_w^*(y,t)= \sum_{i=0}^{n} u_{w,i}^* \, \II(y_i \leq t <y_{i+1})$ to be monotone increasing in $t$. 
This is satisfied for instance for $d_0=F_n$ and the best invariant $d_0^*$, where $u_i=i/n$ and $u_i^*=(i+1)/(n+2)$,
respectively.  We will also have monotonicity  when $w$ is a constant,   since  the target $d_0$ is a cdf and thus monotone and  monotonicity of $d^*_w$ is guaranteed by virtue of the monotonicity of $d_0^*$ established in Theorem \ref{thm.best}. Taken together, the above conditions suggest a strategy in the selection of $d_0$ and $w$ which will
lead to the best invariant estimator being a genuine cdf.
 
\begin{rem}
As in Section 2, for estimating $F$ by $d$ under loss
$L_{w,d_0}(\tau(d),F)$, the best invariant procedure is given by
$\tau^{-1}(d_{w}^*(Y,\cdot))$, for invariant $d_0$ and $w$.
\end{rem}

\begin{thm}
\label{blfminimax} For invariant $d_0$ and $w$, the best invariant
estimator $d_w^*$ of $\tau(F)$ in Theorem \ref{blfbestinvariant} is minimax
under loss (\ref{iblf}) with  $\tau$  being  H\"{o}lder continuous (i.e., $\tau \in \LL(\alpha)$).
\end{thm}
\begin{proof}
The proof is given in Section \ref{thm5}
\end{proof}

We conclude this section by establishing a dominance result that is quite general, and valid for any choice of a target
estimator $d_0$ (invariant or not, with constant risk or not). The only requirement is that the weight function $w$ used
for defining the integrated balanced loss be constant.

\begin{thm}
\label{dominance}
For estimating $F$ under balanced integrated loss $L_{w,d_0}$ in (\ref{iblf}) with constant weight
$w$, i.e., $w(x,t)=\alpha \,(\hbox{say}) \in (0,1)$ for all $(x,t) \in \mathbb{R}^{n+1}$, the estimator 
$\alpha d_0 + (1-\alpha) d_1$ dominates the estimator $\alpha d_0 + (1-\alpha) d_0^*\,$, where $d_1$ is an 
estimator of $F$ which dominates $d_0^*$, the best invariant estimator under integrated squared error loss
$L_0(d,F)=\int_{\RR} (d(t)-F(t))^2 dF(t)$.
\end{thm}
\begin{proof}
The proof is given in Section \ref{thm6}.
\end{proof}

Under integrated squared error loss, Brown (1988) provides such dominating estimators $d_1$ of the best invariant
estimator $d_0^*(y,t)=\sum_{i=0}^{n} (\frac{i+1}{n+2}) \, \II(y_i \leq t <y_{i+1})$.  Also, notice that the dominating
estimators of the above theorem are
necessarily minimax for invariant $d_0$ by virtue of Theorem \ref{blfminimax}.

%
\section{Application to nomination sampling}
%

Consider $n$ observations that come in the form of independent order statistics
that are of the same rank and obtained from independent samples (referred to as \textit{sets}) of  size  $k$. For instance, it could be the case
that the $n$ observations are the maxima of $n$ sets of $k$ i.i.d.~observations, and thus,
i.i.d.~themselves. Such a sampling scheme is generally referred to as \textit{nomination sampling},
a term introduced by Willemain~(1980), and more specifically as maxima-nomination sampling
in the example at hand. For further details,  see  Samawi et al.~(1996), as well as  Jafari Jozani and Johnson~(2011).
In this section, we study  two  examples of nomination sampling:  maxima and median nomination samplings. In Section \ref{case-study},   we discuss using an 
integrated balanced loss function for estimating the distribution of bilirubin levels in the
blood of babies suffering from jaundice, an application previously presented by Sawami and
Al-Sagheer~(2001).

\subsection{Maxima-nomination sampling}

 Suppose  $X=(X_{1},  \ldots, X_n)$  is a maxima
nominated sample of size $n$  with set sizes $k$, so that the $X_i$ are i.i.d.\  observations  with cdf  $F$,   $i = 1,  \ldots, n$.  
The focus here is on estimating the underlying cdf $\tau(F)=F^{1/k}$ using two competing best invariant  estimators (under different
losses). First, using  loss 
\begin{equation}
L_1(d,F) = \int_{\mathbb{R}} \big( d(t)-\tau(F(t)) \big)^2 dF(t),
\label{eq.nom.L1}
\end{equation}
with $\tau(z) = z^{1/k}$,  
Corollary~1(a) implies that the best invariant estimator $d_1^*$ of $\tau(F)=F^{1/k}$ is given by~\eqref{eq.invest}, with optimal
weights
\begin{equation*}
u_{1,i}^*= E[{T_i}^{1/k}] 
= \prod_{j=i}^{n} \left( \frac{j+1}{j+1+1/k} \right),
\end{equation*}
upon adapting the result obtained in \eqref{general-u}. Another approach consists in using loss
\begin{equation}
L_2(d,F) = \int_{\mathbb{R}} \big( d(t)-\tau(F(t)) \big)^2 d\tau(F(t))
= \int_{\mathbb{R}} \big( d(t)-\tau(F(t)) \big)^2  H(F(t))\, dF(t),
\label{eq.nom.L2}
\end{equation}
with $H(z)= \frac 1k z^{\frac 1k -1}$.
Loss $L_2$  differs from $L_1$ as it  considers an integrated distance between $d$ and $\tau(F)$ weighted according to $\tau(F)$ rather than $F$.
Following Remark \ref{rem1},
the best invariant estimator  $d_2^*$ of $\tau(F)$  is given by~\eqref{eq.invest}, with optimal weights
\begin{equation}
u_{2,i}^* 
= \frac{E[ {T_i}^{-(k-2)/k}]}{E[ {T_i}^{-(k-1)/k}]}
= \prod_{j=i}^{n} \left( \frac{j+1/k}{j+2/k} \right),
\label{eq.est2}
\end{equation}
for $i=0,  \ldots, n$.
These estimators will be compared to the MLE of $\tau(F)$ (see~Boyles and Samaniego, 1986), denoted
$d_{\text{MLE}}$, that is also of
the form~\eqref{eq.invest}, but with weights
\begin{equation}
u_{\text{MLE},i} = (i/n)^{1/k}, 
\label{eq.mle}
\end{equation}
for $i=0, \ldots, n$.
We point out  that $d_2^*$ corresponds to the LSE of $\tau(F)$ introduced by Kvam and Samaniego (1993)
when considering the special case of i.i.d.~observations.

\begin{rem} (For the case of minima-nomination sampling, suppose  $X_{1},  \ldots, X_{n}$ are  
independent minima of samples of size $k$ so that $X_i$ are i.i.d.\  observations  with distribution   $F$. The focus here is on estimating the underlying cdf $\tau(F)= 1- (1-F)^{1/k}$.  
Working with  loss functions~\eqref{eq.nom.L1} and~\eqref{eq.nom.L2} to estimate $ \tau(F) = 1-(1-F)^{1/k}$, one can easily   obtain the best
invariant (and minimax) estimators of $\tau(F)$ under losses $L_1$ and $L_2$, with weights
\begin{equation*}
u_{1,i}^* = 1-E[ (1-T_i)^{1/k}], \qquad \text{and} \qquad u_{2,i}^* = 1-\frac{E[ (1-T_i)^{-(k-2)/k}]}{E[(1-T_i)^{-(k-1)/k}]},
\end{equation*}
for $i=0, \ldots, n$.
\end{rem}

\subsection{Median-nomination sampling}

As an  another interesting example, we consider the case of median-nomination sampling of 
Muttlak~(1997).
Assuming
the set size $k$  is odd,  suppose  $X_1,  \ldots, X_{n}$ are independent medians of sets
of size $k$, so that $X_i$ are i.i.d.\ observations  with cdf  $F$. We are interested in estimating the underlying cdf $ \tau(F) = \Psi^{-1}(F)$,  with  \begin{equation*}
\Psi(F)= \sum_{j=(k+1)/2}^{k} \binom{k}{j} F^j (1-F)^{k-j},
\end{equation*}
where   $\Psi$ is the  Beta($\frac{k+1}{2}, \frac{k+1}{2}$)  distribution
function and hence  strictly increasing.   
Under loss $L_1$, the best invariant (and minimax) estimator of $\tau(F)=\Psi^{-1}(F)$ is obtained when
$u_{1,i}^* 
= E[\Psi^{-1}(T_i)]$,
while under loss $L_2$, the best invariant (and minimax) estimator of $\Psi(F)$ is obtained when
\begin{equation*}
u_{2,i}^* = E \left[ \frac{\Psi^{-1}(T_i)}{\Psi'[\Psi^{-1}(T_i)]} \right] \left\{ E \left[ \frac{1}{\Psi'[\Psi^{-1}(T_i)]} \right] \right\}^{-1}.
\end{equation*}
for $i=0, \ldots, n$. The given expectations have to be evaluated numerically.
In this context,
the MLE of $\tau(F)$ is obtained when
\begin{equation*}
u_{\text{MLE},i} = \Psi^{-1}(i/n) \qquad \text{for } i=0, \ldots, n. 
\end{equation*}

\subsection{Simulated examples}

We now study the behaviour of the proposed estimators using simulated minima and median nominated   data where the true distribution $\Phi(\cdot)$ is standard normal. First, suppose 
 $X_1,  \ldots, X_{10}$ are  i.i.d.~maxima nominated samples of size $n=10$  with  cdf $F$, when  the  set size is 
$k=5$. 
It is expected that the maxima-nomination sampling scheme would produce estimators of  the underlying cdf $\tau(F) = F^{1/k}$
that should behave quite well in the upper tail of the estimated distribution.

\begin{figure}[h!]
\begin{center}
\includegraphics[scale=0.65]{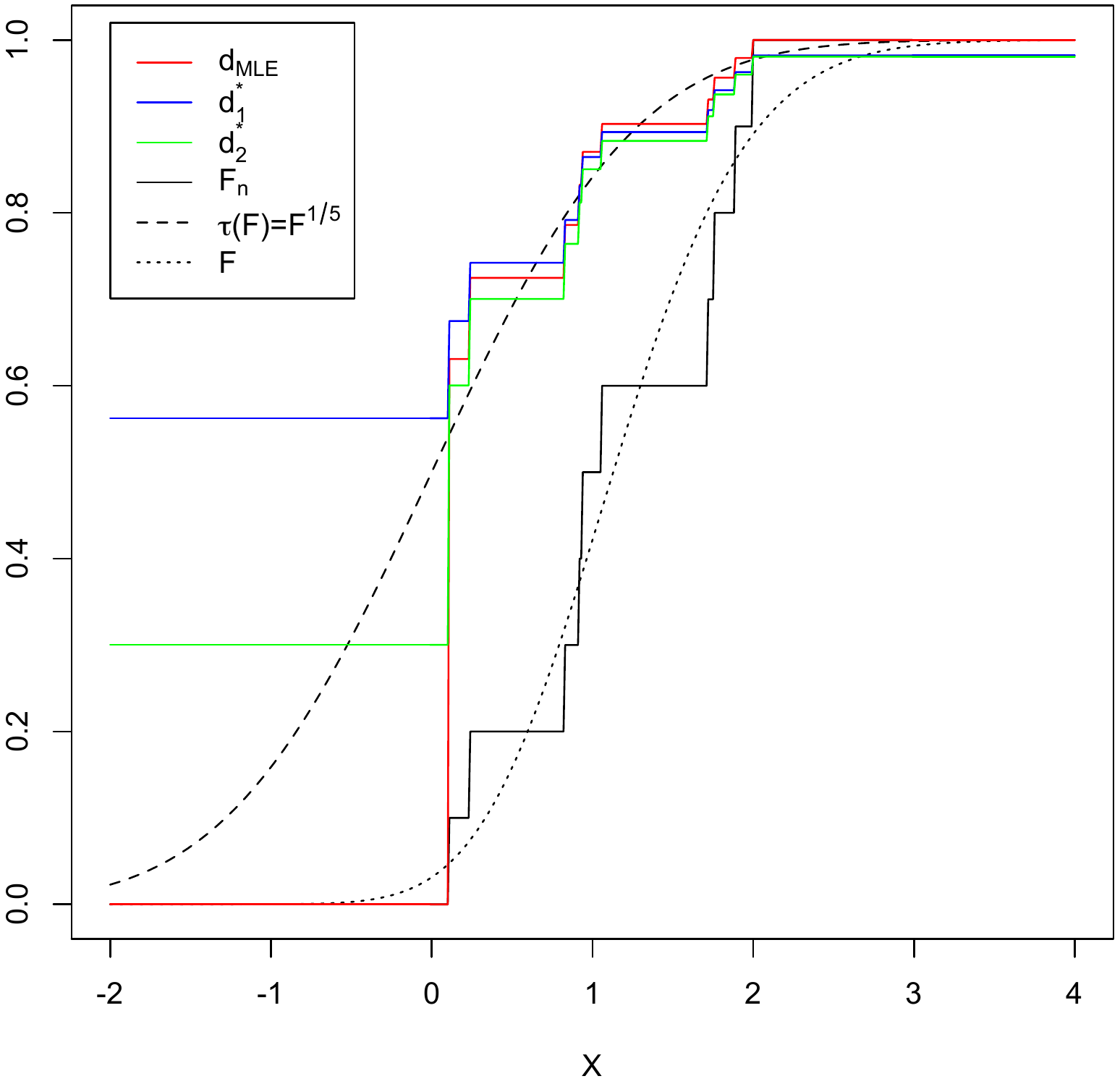}
\caption{Simulated normal maxima-nomination data ($n=10$, $k=5$).}
\label{fig.max5.1}
\end{center}
\end{figure}

\begin{figure}[h!]
\begin{center}
\includegraphics[scale=0.65]{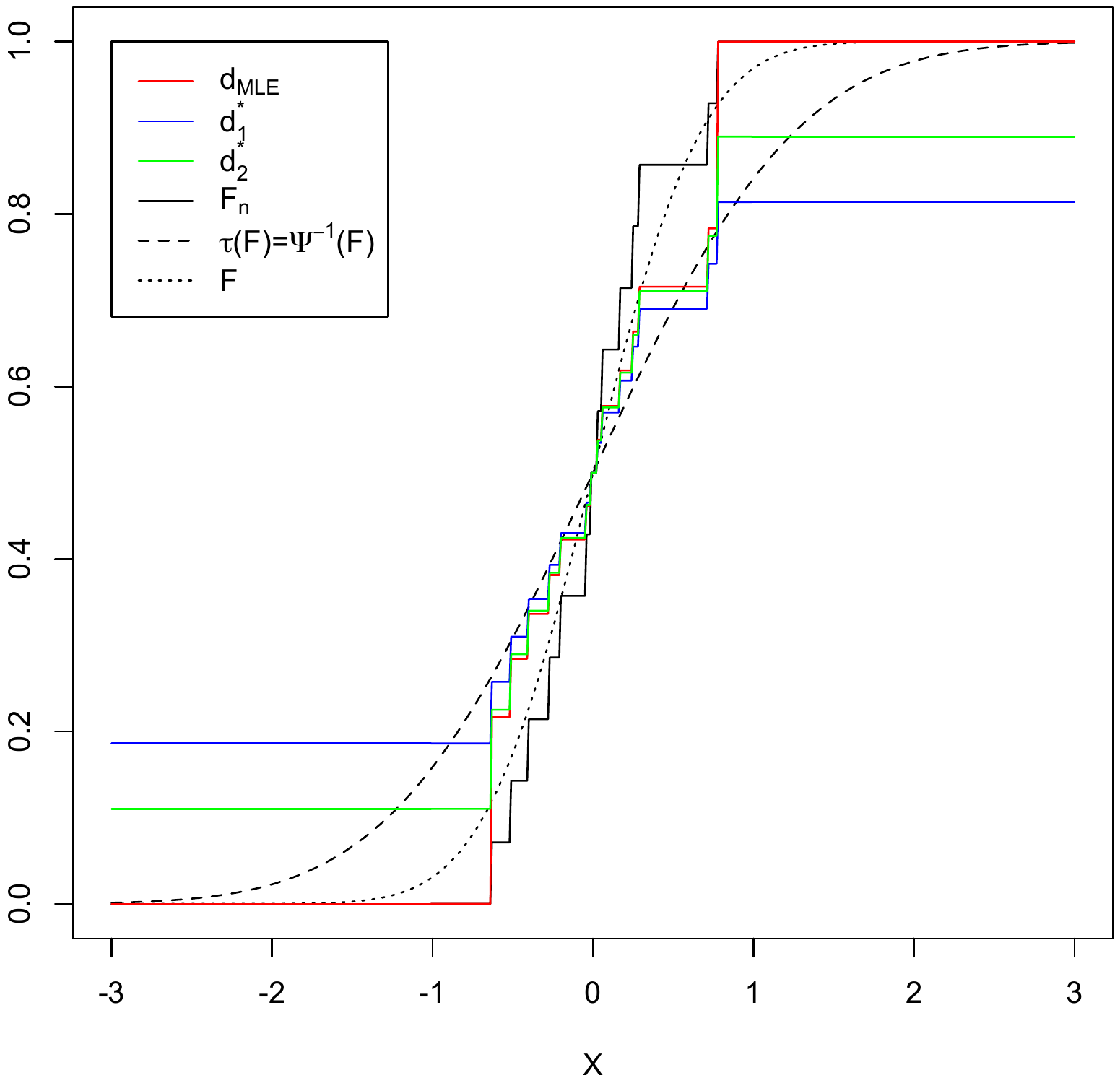}
\caption{Simulated normal median-nomination data ($n=10$, $k=5$).}
\label{fig.med}
\end{center}
\end{figure}

This is confirmed visually
through a quick inspection of Figure~\ref{fig.max5.1}  where all considered estimators
perform quite well in the right tail based on the  maxima nomination sample with $k=5$.  Similar behaviour   was observed in the cases where   $k=3$ and $7$, but results are not reported here. 
In Figure~\ref{fig.max5.1},
it is also very interesting to see how working with $d_2^*$ over $d_1^*$ leads to improved inference.
Indeed, minimizing the integrated distance between $d$ and $\tau(F)$ by weighting that distance with
respect to $\tau(F)$ itself gives a much more sensible estimator in the left tail. This is essentially because
that left tail plays almost no role when weighting the distance with respect to $F$ (which has a much shorter
left tail than $\tau(F)$). As could be expected, the impact of this is particularly important for larger values of $k$.
The empirical distribution function $F_n$ of the raw data is also shown on all graphs, to help with the comparisons.

\begin{table}[h]
\begin{center}
\begin{tabular}{ccccccc}
\toprule
$i$ & 0 & 1 & 2 & 3 & 4 & 5 \\
\midrule
$u_{1,i}^*$ & 0.209 & 0.291 & 0.352 & 0.405 & 0.453 & 0.500 \\
$u_{2,i}^*$ & 0.125 & 0.257 & 0.332 & 0.393 & 0.448 & 0.500 \\
$u_{MLE,i}$ & 0.000 & 0.247 & 0.327 & 0.390 & 0.446 & 0.500 \\
\bottomrule
\end{tabular}
\caption{Weights of minimax estimators in the median-nomination sampling  case}
\end{center}
\end{table}

For median-nomination sampling, we also considered  the case where $n=10$ and $k=5$. 
For all estimators, the values of the weights (obtained from numerical integration, except in the case of the MLE) 
are provided in Table~1 for $i=0, \ldots,5$. The values that are not displayed in the
table can be easily recovered by symmetry of the estimators under the median-nomination sampling (i.e., $u^*_{n-i}= 1-u^*_{i}$).
We note that Samawi and Al-Sagheer (2001) suggested to use $F_n$ to estimate $\tau(F)$
without modification for values of $t$ such that $\tau(F(t)) \simeq 1/2$. Figure~\ref{fig.med}
suggests that this is reasonable, but that both tails are not captured very well
when using this sampling scheme.




\section{A case study}\label{case-study}

Hyperbilirubinemia is a medical condition which commonly affects newborn babies and
that arises when the bilirubin levels in the blood exceed 5 mg/dl. Now, bilirubin's natural
pigmentation typically causes a yellowing of the baby's skin and tissues accompanying
hyperbilirubinemia, which is known as jaundice.
The level at which the concentration of bilirubin in the blood becomes dangerous is considered
to vary between infants, but the effects of bilirubin toxicity can be permanent and include, for instance,
developmental delays and hearing loss.

In a study of bilirubin levels in the blood of babies suffering from jaundice staying in the neonatal
intensive care unit of five hospitals from Jordan, Samawi and Al-Sagheer~(2001) considered data
obtained according to a nomination sampling scheme. It is noted that ranking of the level of bilirubin
in the blood can be done visually by observing the colour of the face, chest and extremities
of babies, as the severity of jaundice is directly related to the concentration of bilirubin in the blood.
This fact is quite important as it allows easy ordering of a small number of sampled babies,
in terms of bilirubin concentration, without having to actually measure those concentrations
by running a blood test, which requires about 30 minutes for completion.

Interest lied mainly in estimating the distribution function  of Bilirubin level in the blood of 
jaundice babies (in mg/dl). Among other things, the authors were interested in recovering
the quantile of order 0.95 of Bilirubin level in the blood of 
jaundice babies. Also, as it is considered that a concentration of 17.65 mg/dl
should not be exceeded to avoid any long term repercussions on a baby's health, they considered
 the order of the quantile associated with 17.65  as another quantity of interest. Note that both of these quantities are related to
the right tail of the underlying distribution, suggesting that a maxima-nomination sampling scheme is appropriate.

We here consider the estimation of the underlying cdf $\tau(F)$ from the $n=14$ maxima listed in Table 4.1 of Samawi
and Al-Sagheer~(2001). In Figure~\ref{fig.bilirubin}, we have displayed the minimax estimator
$d_2^*$ given in~\eqref{eq.est2}, the MLE of $\tau(F)$ given in~\eqref{eq.mle} and the minimax estimator
obtained under the balanced loss
\begin{equation*}
 \int_{0}^{\infty} \big\{ w(t) \big( d_0(t)-\tau(F(t)) \big)^2 + (1-w(t)) \big( d(t)-\tau(F(t)) \big)^2
\big\} d\tau(F(t)),
\end{equation*}
where $\tau(F(t)) = \{F(t)\}^{1/5}$, the target estimator $d_0$ is the MLE of $\tau(F)$ and the weight function $w$
is such that
\begin{equation*}
w(t) = \begin{cases}
1/2 & \text{if } t < Y_n \\
1 & \text{if } t \geq Y_n.
\end{cases}
\end{equation*}
As in Theorem \ref{blfbestinvariant}, we obtain the best invariant estimator as follows 
\begin{equation*}
u_i^* = \begin{cases}
\frac{1}{2} (u_{2,i}^* + u_{\text{MLE},i}) & \text{for } i=0,1,\ldots, n-1 \\
u_{\text{MLE},n} & \text{for } i=n
\end{cases},
\end{equation*}
with $u_{2, i}^*$ given in \eqref{eq.est2} with $k=5$. 
The choice of the weighting function $w$ along with the choice $Y_0=0$ (i.e., $a=0$ and $b=\infty$, see Remark \ref{genuine}) forces $d^*$ to be a genuine distribution function.  
\begin{figure}[h]
\begin{center}
\includegraphics[scale=0.65]{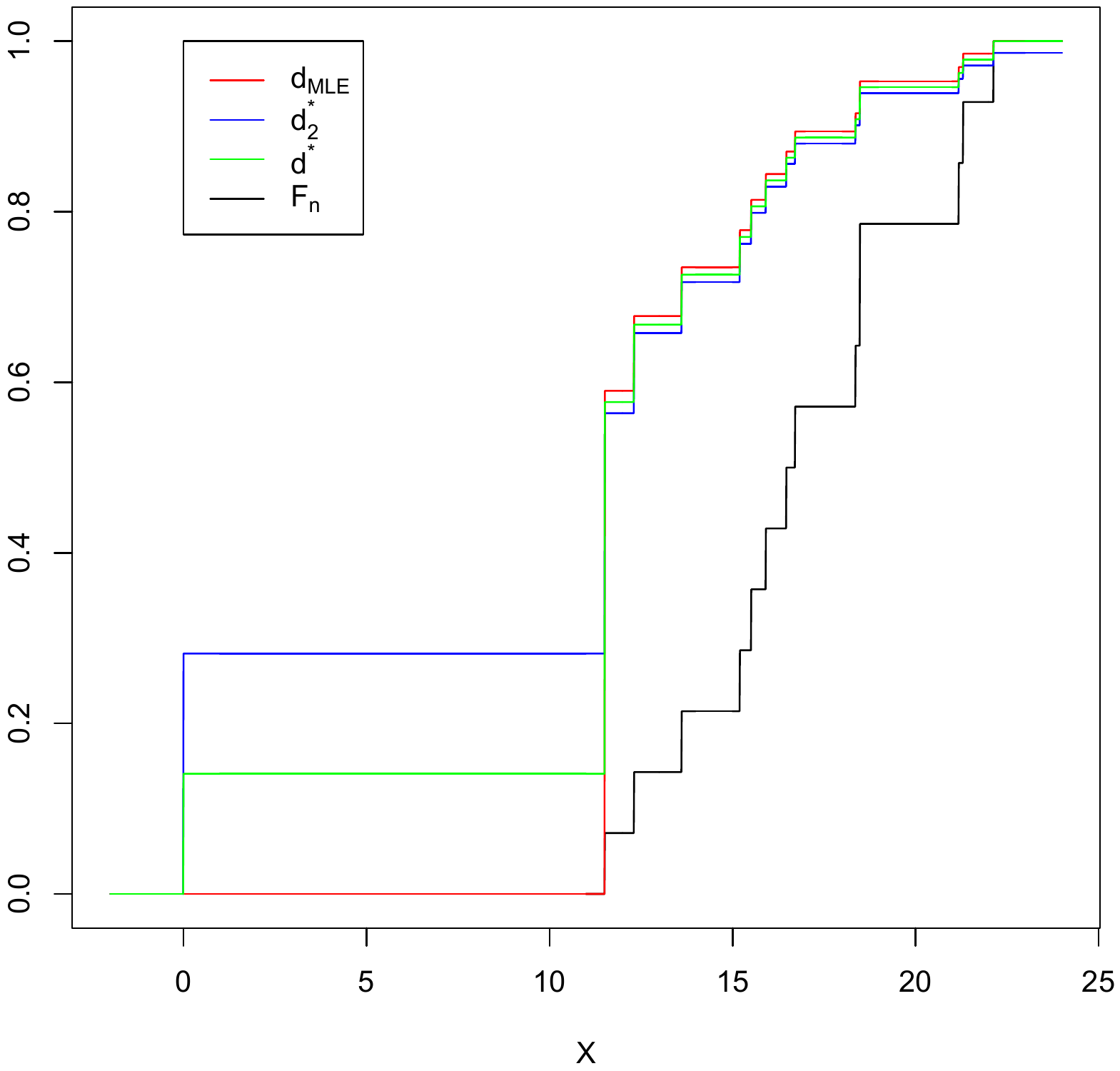}
\caption{Bilirubin concentrations data ($n=14$, $k=5$).}
\label{fig.bilirubin}
\end{center}
\end{figure}
Figure~\ref{fig.bilirubin} shows the impact of using a balanced loss approach. Again,
the difference between the estimators is most important in the left tail of the estimated distribution.
However, this is the most important aspect to consider here as all the considered estimators
seem to perform reasonably well in capturing the right tail of $\tau(F)$ in the normal example seen earlier. But,
when estimating the left tail of $\tau(F)$, the MLE clearly needs to be
improved. Using the suggested balanced loss is one way to accomplish this, while leading
to an estimated distribution that is bona fide.

%
\section{ Concluding remarks} 
%

Our findings relate to the estimation of a continuous distribution function $F$, as well as meaningful functions $\tau(F)$.  For the large class of loss functions
$L_{\rho,\tau}(d,F)$, as well as weighted versions (Section 2.3), we have obtained best invariant estimators (Section 2.1) and established their minimaxity (Section 2.2) for H\"{o}lder continuous $\tau$'s and strict bowl-shaped $\rho$ with a bounded derivative.  For identity $\tau$, our minimax result extends previously established results.
For non-identity $\tau$, the results are novel and apply as well for the minimaxity of estimators of $\tau(F)$.  Many new cases are covered such as integrated $L^p$ ($p \geq 1$) losses and integrated ratio losses of the form $\int_{\mathbb{R}} \rho_0(\frac{d(t)}{F(t)}) \, dF(t)$.  
We have also remarked upon the (known) fact that best invariant minimax solutions often fail to be genuine distribution functions, and expanded upon corresponding adjustments (Remark \ref{genuine}).  In Section 3, we introduced and motivated the use of integrated balanced loss functions which combine the criteria of an integrated distance as above between a decision $d$ and $F$, with the proximity of $d$ with a target estimator $d_0$.  Moreover, we have shown how the risk analysis of procedures under such an integrated balanced loss relates to a dual risk analysis under an ``unbalanced'' loss, and we have derived best invariant estimators, minimax estimators, risk comparisons, dominance and inadmissibility results.   
We believe that the further development of estimating procedures via integrated balanced loss functions
is of interest and appealing.  For instance, enough flexibility is built in to select a model based or fully parametric target estimator $d_0$, 
assuming for instance a normal distribution function $F$, and obtain compromise efficient procedures such as Theorem \ref{blfbestinvariant}'s $d_w^*$.

\section{Appendix}

\subsection{Proof of Theorem \ref{thm.best}}\label{thm1}
Following arguments of Ferguson (1967, Section~4.8),  the risk of $d$ in estimating $F$, for any invariant
estimator of the form~\eqref{eq.invest} and under the loss~\eqref{lossG}, may be decomposed as
\begin{eqnarray}
\nonumber R_{\rho,\tau}(d,F) &=& E_{F} \big[ \int_{\mathbb{R}}  \rho\big(\tau(\sum_{i=0}^n u_i \II(Y_i \leq t < Y_{i+1}) )- \tau(F(t))\big) \, dF(t)  \big] \\
\nonumber
\, &=& E_{F} \big[ \int_0^1  \rho\big((\sum_{i=0}^n \tau(u_i) \II(Y_i \leq F^{-1}(t) < Y_{i+1}) )- \tau(t)\big) \, dt  \big] \\
\nonumber
\,&=&  \sum_{i=0}^n    \int_0^1 \rho(\tau(u_i) - \tau(t)) \,
E_{F} (\II(F(Y_i) \leq t < F(Y_{i+1})) ) \, d t 
 \\
\label{s}
\, &=&   \sum_{i=0}^n    \int_0^1 \rho(\tau(u_i) - \tau(t)) \,
\binom{n}{i} \, t^i \, (1-t)^{n-i}\, d t \,  \\
\label{ss}
\, &=&  \frac{1}{n+1} \; \sum_{i=0}^n  \int_0^1 \rho(\tau(u_i) -\tau(t)) \, f_{T_i}(t) \, dt.
\end{eqnarray}
With the minimization problem now reducing to minimizing every element of the above sum in (\ref{s}), the results follow immediately. Also, $\tau(u_i^*)$ minimizes (\ref{s}) in $\tau(u)$ and hence satisfies the equation  $B_i(\tau(u_i^*))=0$, with
$ B_i(\tau(u)) = \int_0^1 \rho'(\tau(u) - \tau(t)) \,
\binom{n}{i} \, t^i \, (1-t)^{n-i}\, d t \,. $  Since $\rho'(\tau(0) - \tau(t)) < 0$ for all $t \in (0,1)$ given the conditions on $\rho$ and $\tau$, we have $B_0(\tau(0)) < 0$, whence $u_0>0$.  Similarly, we have $B_n(\tau(1))>0$ and $u_n<1$.  
The monotonicity property of the $u^*_i$'s follows from  complete class theorems for monotone procedures such as those provided by Karlin and Rubin (1956) or Brown, Cohen and Strawderman (1976).  Indeed, these results apply for families of densities with strict increasing monotone likelihood ratio, such as $\hbox{Bin}(n,p)$ distributions with $p \in [0,1]$, and for the problem of estimating $p$ under strict bowl-shaped loss $L(d,p)$.


\subsection{Proof of Lemma \ref{thm.eps}}\label{lem2}
Let $\delta,s > 0$ and set 
$$B_{s} = \{ (\Y, t) : | d(\Y; t) -d_0(\Y; t)| \geq s \}\,.$$
Using Lemma~\ref{lemma.yuchow}, there
exists $P_0$ (with associated distribution function $F_0$) and an estimator $d_0 \in \mathcal{I} $ such that 
\begin{equation}
P_0^{n+1} (B_{s}) \leq \delta\,.
\label{eq.probyu}
\end{equation}
Now, (i) the triangular inequality, (ii) the boundedness of $\rho'$, and (iii) the  
H\"{o}lder continuity assumption enable us to write
\begin{align*}
\big| R(d, F_0) -R(d_0,F_0) \big|
& \leq \EE_{F_0} \left[ \int_{\RR} \Big| \rho(\tau(d(\Y;t)) - \tau(F_0(t))) \,-\,
\rho(\tau(d_0(\Y;t)) - \tau(F_0(t))) \Big| dF_0(t) \right] \\
&  \leq (\sup_u \rho'(u)) \; \EE_{F_0} \left[ \int_{\RR} \Big| \Big( \tau \big( d(\Y; t) \big) - \tau \big( d_0(\Y; t) \big) \Big)
\Big| dF_0(t) \right] \\
&  \leq M \, (\sup_u \rho'(u)) \; \EE_{F_0} \left[ \int_{\RR} \Big| \tau \big( d(\Y; t) \big) - \tau \big( d_0(\Y; t) \big) \Big|^{\alpha} \; 
dF_0(t) \right].
\end{align*}
Making use twice of   Jensen's inequality for concave functions (i.e., $|z|^{\alpha}$ with $\alpha \leq 1$) yields
\begin{align}
\big| R(d,F_0) -R(d_0,F_0) \big| & \leq  M \, (\sup_u \rho'(u)) \left( \EE_{F_0} \left[ \int_{\RR} \Big| d(\Y; t) - d_0(\Y; t) \Big| dF_0(t) \right] \right)^{\alpha},
\label{eq.eps5}
\end{align}
Using the fact that $|d(\y; t) -d_0(\y; t)| \leq 1$ for all $t$, we obtain with (\ref{eq.probyu})
\begin{align*}
\EE_{F_0} \left[ \int_{\RR} \Big| d(\Y; t) - d_0(\Y; t) \Big| dF_0(t) \right] & = \int_{B_s} \big| d(\y; t) -d_0(\y; t) \big| dF_0^{n+1}(\y,t)  + \int_{B_s^c} \big| d(\y; t) -d_0(\y; t) \big| dF_0^{n+1}(\y, t) \\
& \leq P_0^{n+1}(B_s) + s \, P_0^{n+1} (B_s^c) \leq \delta +s.
\end{align*}
Finally, substituting this into ~\eqref{eq.eps5} and selecting $\delta, s$ such that  $\delta +s = {\epsilon}^{1/\alpha}\{ M (\sup_u \rho'(u))\}^{-1/\alpha}$, we obtain
$|R(d,F_0) -R(d_0,F_0)| \leq \epsilon$, as desired.

\subsection{Proof of Theorem \ref{thm.minimax}}\label{thm2}
We start by noting that Lemma~\ref{thm.eps} implies that, given $d \in \A$ and $\epsilon>0$, 
there exists $F_0 \in \FF$ and an invariant estimator $d_0 \in \mathcal{I}$ such that
$\big| R_{\tau}(d,F_0)- R_{\tau}(d_0,F_0) \big| \leq \epsilon$,
implying, in turn, that
\begin{equation*}
\sup_{F \in \FF} R_{\tau}(d^*,F) = R_{\tau}(d^*,F_0) \leq R_{\tau}(d_0,F_0) \leq R_{\tau}(d,F_0)+\epsilon \leq \sup_{F \in \FF} R_{\tau}(d,F) + \epsilon,
\end{equation*} 
given that $d^*$ is the best invariant estimator under loss~\eqref{eq.loss} with constant risk.
Since $d$ and $\epsilon$ are both arbitrary, the stated result follows.


\subsection{Proof of Theorem \ref{minimax-weighted} }\label{thm3}
For the best invariant property, proceeding as in the proof of Theorem \ref{thm.best}
yields the result.  For instance, equation (\ref{s}) becomes $$\sum_{i=0}^n w_i   \int_0^1 \rho(\tau(u_i) - \tau(t)) \,
\binom{n}{i} \, t^i \, (1-t)^{n-i}\, d t \,,$$
and it is clearly seen that the minimization is handled irrespectively of the weights $w_i$'s.  For the minimaxity, the developments of Section 2.2 go through by simply bounding $w_n(\cdot)$ by $1$.


\subsection{ Complementary developments and proof of Lemma \ref{99}}\label{lem3}
The representations below are used in Lemma \ref{99} and
generalize Lemma 1 of Jafari Jozani, Marchand and Parsian (2006).
The general context is one of estimating a parameter $\theta$ for
the model $Z \sim F_{\theta}$ with loss
\begin{equation}
\label{blf}
L_{\omega(\cdot),\delta_0}(\delta,\theta)= q(\theta)\, w(z) \,(\delta-\delta_0)^2 \,+ \,q(\theta) (1-w(z)) \, (\delta-\theta)^2,
\end{equation}
where $w(\cdot) \in [0,1]$, $q(\cdot)>0$, and $\delta_0$ is a
target estimator of $\theta$. Under loss (\ref{blf}), it is easy
to check that
\begin{equation}
\label{blfrep1}
L_{\omega(\cdot),\delta_0}(\delta_0+(1-w)g,\theta)= \,q(\theta)\, w(1-w) \,(\delta_0-\theta)^2 \,+ \,q(\theta) (1-w)^2 \, (\delta_0+g-\theta)^2.
\end{equation}
We hence obtain that the risk of the estimator $\delta_0(Z)+(1-w(Z))\,g(Z)$ under loss $L_{\omega(\cdot),\delta_0}$ is decomposable as
\begin{eqnarray}
\label{blfrep2}
\nonumber R_{\omega(\cdot),\delta_0}(\delta_0(Z)+(1-w(Z))g(Z),\theta)&=& E[\,q(\theta)\, w(Z)(1-w(Z)) \,(\delta_0(Z)-\theta)^2] \\
\,&+&
 E[\,q(\theta) (1-w(Z))^2 \, (\delta_0(Z)+g(Z)-\theta)^2]\,,
\end{eqnarray}
i.e., the sum of the risks of $\delta_0(Z)$ and $\delta_0(Z)+g(Z)$
with respect to the weighted squared error losses $q(\theta) \, w(z)(1-w(z)) (\delta-\theta)^2$
and $q(\theta)\, (1-w(z))^2 \, (\delta-\theta)^2$, respectively.
Since the former of these risks does not depend on $g(\cdot)$,
we have an equivalence between the performance of the estimator
 $\delta_0(Z)+(1-w(Z))\,g(Z)$ under balanced loss (\ref{blf})
 and the estimator $\delta_0(Z)+\,g(Z)$ under the second of these weighted (and unbalanced) losses.
  This observation was put forward at the outset of the paper by Jafari Jozani,
   Marchand and Parsian (2006) for the particular case where $w(\cdot)$ is constant
   and they pursued with various connections between the balanced loss and unbalanced
   loss problems as well as applications.  A redeployment of their analysis for non-constant
   weight functions $w(\cdot)$ is available with the above decomposition and of interest.
   Now, to conclude, expression (\ref{blfrep1}) is used in Lemma \ref{99} with for fixed $(x,t) \in \mathbb{R}^{n+1}$ with $X=Z$, $\theta=\tau(F(t))$,
   $q(\cdot) =1$ $\delta_0=d_0(x,t)$, $\delta_0+(1-w)g = d_0(x,t)+(1-w(x))\, g(x,t)$.
   
   To prove Lemma \ref{99},  expanding (\ref{iblf}), we have $L_{w,d_0}
(d_0(x,t)+(1-w(x,t) g(x,t), F\,)=$
$$
\int_{\RR} \{w(x,t)(1-w(x,t)) \,[d_0(x,t)-\tau(F(t))]^2 \, + \,
(1-w(x,t))^2  \,[d_0(x,t)+g(x,t)-\tau(F(t))]^2 \, \} \,
 \,dF(t).
$$  The result thus
follows at once from (\ref{blfrisk}).

\subsection{Proof of Theorem \ref{blfbestinvariant}}\label{thm4}
 First, observe that $d_0^*$ is the best
invariant procedure under loss $L_0$, and thus for risk $R_{H_2}$ by virtue of
Theorem \ref{minimax-weighted} as $w$ is invariant.
 Now, under the assumptions on $d_0$ and $w$, we see from
(\ref{blfrep}) that the risk $R_{w,d_0}$ of invariant
estimators is constant with the optimal choice of $g$
arising for $d_0+g=d_0^*$, which gives $d_{w}^* = d_0 + (1-w)
(d_0^*-d_0) = wd_0 + (1-w) d_0^*$. 


\subsection{Proof of Theorem \ref{blfminimax}}\label{thm5}
 In representation (\ref{blfrep}), observe
that the first term $R_{H_1}(d_0(X,\cdot),F)(=C)$ is constant in
$F$ since $d_0$ is invariant by assumption.  Furthermore, Theorem
\ref{minimax-weighted} tells us, with the choice $\rho(z)=z^2$, that $\sup_F \{R_{H_2}(d_0+g,F)\} \geq \sup_F
\{R_{H_2}(d_0^*,F)\}$ for all $g$.  We hence obtain for any
estimator $d_0 + (1-w) \, g \in
\cal{A}$:
\begin{eqnarray*}
\sup_{F \in \FF}\, \{R_{w,d_0}(d_0+(1-w)g,F \}&=& C + \sup_{F \in \FF} \{R_{H_2}(d_0+g,F)\}
\\ \,& \geq & C+ \sup_{F \in \FF} \{R_{H_2}(d_0^*,F) \} \\
\, &=& \sup_{F \in \FF}\, \{R_{w,d_0}(d_0+(1-w)(d_0^*-d_0),F \} \\
\, &=& \sup_{F \in \FF}\, \{R_{w,d_0}(d_w^*,F \}\,,
\end{eqnarray*}
which yields the result.


\subsection{Proof of Theorem \ref{dominance}}\label{thm6}
This follows directly from expressing the difference in risk of the two estimators as
\begin{equation*}
R_{w,d_0}(\alpha \,d_0 + (1-\alpha) \, d_0^*,F) - R_{w,d_0}(\alpha d_0 + (1-\alpha) d_1,F) = (1-\alpha)^2 \{R_0(d_0^*,F)
-R_0(d_1,F)\} \geq 0,
\end{equation*}
for all $F$, with strict inequality for some.  

%
%

\section*{Acknowledgments}
All three authors  gratefully acknowledge the research support
of the Natural Sciences and Engineering Research Council of Canada.

%
%


\begin{thebibliography}{99}


\bibitem{aggar}
Aggarwal, O.P.~(1955).
{\it Some minimax invariant procedures for estimating a cumulative distribution function.}
Ann. Math. Stat.~\textbf{26}, 450--463.

\bibitem{boyles}
Boyles, R.A.~and Samaniego, F.J.~(1986).
{\it Estimating a distribution function based on nomination sampling}.
Journal of the American Statistical Association {\bf 81}, 1039--1045.

\bibitem{bcs1976}
Brown, L.D, Cohen, A.~and Strawderman, W. E.~(1976).
\textit{A complete class theorem for strict monotone likelihood ratio with applications}.
Ann. Stat.~\textbf{4}, 712--722.

\bibitem{Brown}
Brown, L.D.~(1988).
\textit{Admissibility in discrete and continuous invariant nonparametric estimation problems and in their
multinomial analogs}. Ann. Stat.~\textbf{16}, 1567--1593.

\bibitem{cohenkuo}
Cohen, M.P.~and  Kuo, L.~(1985).
\textit{The admissibility of the empirical distribution function}.
Ann. Stat.~\textbf{13}, 262--271.

\bibitem{dvoretzky}
Dvoretzky, A., Kiefer, J.~and  Wolfowitz, J.~(1956).
\textit{Asymptotic minimax character of the sample distribution function and of the classical
multinomial estimator}.
Ann. Math. Stat.~\textbf{27}, 642--669.

\bibitem{ferguson}
Ferguson, T.S.~(1967). \textit{Mathematical Statistics: A Decision Theoretic Approach.} Academic, New York.

\bibitem{friedman}
Friedman, D., Gelman, A.~and Phadia, E.~(1988).
\textit{Best invariant estimator of a distribution function under the Kolmogorov-Smirnov loss function}.
Ann. Stat.~\textbf{16}, 1254--1261.

\bibitem{Jafari Jozani}
Jafari Jozani, M.  and  Marchand, \'E. (2007).
{\it Minimax estimation of constrained parametric functions for discrete family of distributions}.
Metrika,  \textbf{66}, 151--160.

\bibitem{jjj}
Jafari Jozani, M.~and Johnson, B.C.~(2012).
\textit{Randomized nomination sampling for finite populations}.
Journal of Statistical Planning and Inference~\textbf{142}, 2103--2115.

\bibitem{JMP}   Jafari Jozani, M., Marchand, \'E.~and
Parsian, A.~(2006). {\it On estimation with weighted balanced-type loss function}.
Statistics \&  Probability  Letters {\bf 76}, 773--780.

\bibitem{kr1956}
Karlin, S.~and Rubin, H.~(1956).
\textit{The theory of decision procedures for distributions with monotone likelihood ratio}.
Ann. Math. Stat.~{\bf 27}, 272--299.

\bibitem{kvam}
Kvam, P.H.~and Samaniego, F.J.~(1993).
{\it On estimating distribution functions using nomination samples}.
Journal of the American Statistical Association {\bf 88}, 1317--1322.

\bibitem{mohammadi}
Mohammadi, L.~and van Zwet, W.R.~(2002).
{\it Minimax invariant estimator of a continuous distribution function under entropy loss}.
Metrika \textbf{56}, 31--42. 

\bibitem{muttlak}
Muttlak, H.A.~(1997).
{\it Median ranked set sampling}. Journal of Applied Statistic Science {\bf 6}, 245--255.

\bibitem{ningxie}
Ning, J.~and  Xie, M.~(2007).
{\it Minimax invariant estimation of a continuous distribution function under LINEX loss}.
J. Syst. Sci. Complex \textbf{20}, 119--126.

\bibitem{nostrom}
Norstr{\o}m, J.G.~(1996).
{\it The use of precautionary loss function in risk analysis}.
IEEE Trans. Reliab.~\textbf{45}, 400--403.

\bibitem{phadia}
Phadia, E.~(1973).
\textit{Minimax estimation of a cumulative distribution function}.
Ann. Stat.~\textbf{1}, 1149--1157.

\bibitem{samawi01}
Samawi, H.M.~and Al-Sagheer, O.A.M.~(2001).
{\it On the estimation of the distribution function using extreme and median ranked set sampling}.
Biometrical Journal {\bf 43}, 357--373.

\bibitem{samawi96}
Samawi, H.M., Ahmed,M.~and Abu-Dayyeh,W.~(1996).
{\it Estimating the population mean using extreme ranked set sampling}. 
Biometrical Journal {\bf 38}, 577--586.

\bibitem{schabe}
Sch\"abe, H.~(1991).
{\it  Bayes estimates under asymmetric loss}.
IEEE Trans. Reliab.~\textbf{40}, 63--67.

\bibitem{baran}
St\k{e}pie\'n-Baran, A.~(2010).
{\it Minimax invariant estimator of a continuous distribution function under a general loss function}.
Metrika, {\bf 72}, 37--49.

\bibitem{WellTiwari}
Wells, M.T.~and Tiwari, R.C.~(1990). \textit{Estimating a distribution function based on minima-nomination sampling}. 
Topics in statistical dependence,  IMS Lecture Notes Monogr. Ser.~\textbf{16}, Inst. Math. Statist., Hayward, CA, 471--479. 

\bibitem{yu89}
Yu, Q.~(1989).
\textit{Methodology for the invariant estimation of a continuous distribution function}.
Ann. Inst. Stat. Math.~\textbf{41}, 503--520.




\bibitem{yu92}
Yu, Q.~(1992).
{\it Minimax invariant estimator of a continuous distribution function}.
Ann. Inst. Stat. Math.~\textbf{44}, 729--735. 

\bibitem{yuphadia}
Yu, Q.~and  Phadia, E.~(1992).
\textit{Minimaxity of the best invariant estimator of a distribution function under the Kolmogorov-Smirnov loss}.
Ann. Stat.~\textbf{20}, 2192--2195.

\bibitem{yuchow}
 Yu, Q.~and Chow, M.~(1991).
 \textit{Minimaxity of the empirical distribution function in invariant estimation}.
 Ann. Stat.~\textbf{19}, 935--951.

\bibitem{willemain}
Willemain, T.R.~(1980).
{\it Estimating the population median by nomination sampling}.
Journal of the American Statistical Association {\bf 75}, 908--911.

\end{thebibliography}
\end{document}